\documentclass[12pt,a4paper]{amsart}

\usepackage{amscd}
\usepackage{graphicx}
\usepackage{amssymb,amsfonts,latexsym}
\usepackage{geometry}

\input xy
\xyoption {all}

\newtheorem{thm}{Theorem}[section]

\newtheorem{cor}[thm]{Corollary}
\newtheorem{lem}[thm]{Lemma}
\newtheorem{prop}[thm]{Proposition}
\theoremstyle{definition}

\theoremstyle{remark}
\newtheorem{rem}[thm]{Remark}

\numberwithin{equation}{section}
\newcommand{\dQ}{\mathbb{Q}}
\newcommand{\dC}{\mathbb{C}}

\newcommand{\dZ}{\mathbb{Z}}

\newcommand{\cP}{\mathcal{P}}

\newcommand{\fp}{\frak{p}}

\long\def\forget#1\forgotten{}

\newcommand{\nn}{{\mathbb{N}}}
\newcommand{\qq}{{\mathbb{Q}}}
\newcommand{\rr}{{\mathbb{R}}}
\newcommand{\zz}{{\mathbb{Z}}}

\begin{document}

\author{Joachim K$\ddot{\rm o}$nig}
\address{Department of Mathematics\\
Technion -- Israel Institute of Technology\\
Haifa, 32000\\
Israel }
\email{koenig.joach@technion.ac.il}

\author{Daniel Rabayev}
\address{Department of Mathematics\\
Technion -- Israel Institute of Technology\\
Haifa, 32000\\
Israel }
\email{daniel.raviv@gmail.com}

\author{Jack Sonn}
\address{Department of Mathematics\\
Technion -- Israel Institute of Technology\\
Haifa, 32000\\
Israel }
\email{sonn@math.technion.ac.il}

\title[ ]{Galois realizations with inertia groups of order two}

\maketitle
\begin{abstract}
There are several variants of the inverse Galois problem which
involve restrictions on ramification.  In this paper we give
sufficient conditions that a given finite group $G$ occurs infinitely
often as a Galois group over the rationals $\dQ$ with all nontrivial
inertia groups of order $2$.  Notably any such realization of $G$ can be translated up to a quadratic field over which the corresponding realization of $G$ is unramified.
 The sufficient conditions are imposed on a parametric polynomial with
Galois group $G$--if such a polynomial is available--and the infinitely many realizations come from infinitely many specializations of the parameter in the polynomial.  This will be applied to the three finite simple groups $A_5$, $PSL_2(7)$ and $PSL_3(3)$.  Finally, the applications to $A_5$ and $PSL_3(3)$ are used to prove the existence of infinitely many optimally intersective realizations of these groups over the rational numbers (proved for $PSL_2(7)$ by the first author in \cite{jk}).
\end{abstract}

\section{Introduction}\label{sec:exp2}
\vskip 2em
There are several variants of the inverse Galois problem which
involve restrictions on ramification.  For example, given a finite
group  $G$ and a finite set  $S$ of primes, does there exist a
Galois extension of the rationals $\dQ$ unramified outside $S$ with
Galois group $G$?  A weaker  form of this problem is the "minimal
ramification problem", which asks what is the minimal number of
ramified primes in a Galois realization of a given finite group $G$
over $\dQ$.  In this paper, instead of restrictions on the set of ramified primes or on the number of ramified primes, we are interested in restrictions on the inertia
groups of the ramified primes.  For example, a result of Van der Waerden \cite{VdW} says
that if the discriminant of a monic irreducible polynomial $f(x)\in \dZ[x]$ of degree $n$ is
divisible by a prime $p$ to the exact power $1$, then the inertia
group of a prime divisor of $p$ in the splitting field of $f(x)$ is
of order $2$ and generated by a transposition.  Furthermore, if the discriminant is squarefree, then the Galois group is the symmetric group $S_n$ (and from the preceding part, all the nontrivial inertia groups have order $2$ and generated by a transposition (see also Kondo \cite
{ko}).  Given a finite group $G$ which is realizable over $\dQ(t)$ as a Galois group of a regular extension by a parametric polynomial $f(t,x)$, if certain conditions on this polynomial hold, then we prove that there are infinitely many specializations of $t$ for which, in the corresponding splitting field, all nontrivial inertia groups have order $2$.  In addition to an illustrative example $G=S_3$, we give three nontrivial examples of such polynomials, all with nonabelian simple Galois groups: $A_5$, $PSL_2(7)$ and $PSL_3(3)$.

\vskip .5em
An important application of this result is the following:
given any realization $K/\dQ$ of a finite group $G$ with all nontrivial inertia groups of order $2$, there exists a quadratic extension $E/\dQ$ such that $G(KE/E)\cong G$ and $KE/E$ is unramified at all primes (finite and  infinite).
\vskip .5em
Finally, we give an application of this result to intersective polynomials.  A monic polynomial in one variable with rational integer coefficients is called (nontrivially){\it intersective} if it has a root modulo $m$ for all positive integers $m$, and it has no rational root.    Let $G$ be a finite noncyclic group and let $r(G)$ be the smallest number of irreducible factors of an intersective
 polynomial with Galois group $G$ over $\dQ$.
There is a group-theoretically defined lower bound for $r(G)$, given by the
 smallest number $s(G)$ of proper subgroups of $G$ having the property that the union of the
 conjugates of those subgroups is $G$ and their intersection is trivial.
 We call an intersective polynomial $f(x)$ with Galois group $G$ over $\dQ$ {\it optimally intersective for}  $G$ if the above lower bound $s(G)$ for $r(G)$ is attained, i.e. $f(x)$ is the product of $s(G)$ irreducible factors.  Accordingly, a Galois extension $K/\dQ$ with Galois group $G$ is called an {\it optimally intersective realization of} $G$ if it is the splitting field of a polynomial which is  optimally intersective for  $G$.  We apply the order two inertia result above to
$A_5$ and $PSL_3(3)$ to prove the existence of infinitely many optimally intersective Galois realizations of $A_5$ and $PSL_3(3)$ over $\dQ$.  This has already been done by the first author for $PSL_2(7)$ \cite {jk}.  Such results for $D_5$ and $A_4$ have appeared earlier \cite{spearman1}, \cite{spearman}.

\vskip 0.5em
We are grateful to Danny Neftin for valuable discussions.

\section{order two inertia}\label{sec:order2}
In what follows we use the following immediate generalization of (a special case of) a theorem of Schinzel \cite{sch}:
\vskip 0.5em
\begin{lem}\label{thm:schinzel}
Let $F_1,..., F_n\in \zz[t]$ and assume that $gcd(F_1,...,F_n)=:N\in \nn$. Furthermore assume that the only primes dividing $gcd(F_1(t_0),..., F_n(t_0))$ for all $t_0\in \zz$ simultaneously, are the prime divisors of $N$.
Then there exists an arithmetic progression of elements $t_0\in \zz$ such that
$gcd(F_1(t_0),..., F_n(t_0))$ is  divisible only by prime divisors of $N$.
\end{lem}
\begin{proof}
The polynomials $G_i:=F_i/N$ lie in the UFD $\zz[t]$ by hypothesis, $gcd(G_1,...,G_n)=1$, and no prime divides $gcd(G_1(t_0),..., G_n(t_0))$ for all $t_0\in \zz$ simultaneously.  Then by the theorem (and proof) of Schinzel \cite [Thm. 1, Cor. 1,  p. 241]{sch}, applied to the $G_i$, the lemma follows.

\end{proof}

Let $f(t,x)\in \dZ[t,x]$ be  monic and separable of degree $n>2$ with respect to $x$.
Let $f'(t,x)\in \dZ[t,x]$ denote the derivative of $f(t,x)$
with respect to $x$, and let $D(t)=D(f(t,x))$ denote the discriminant of $f(t,x)$  with
respect to $x$.  Note that $D(t) \in \dZ[t]$.  (Indeed, $D(t)$ is a symmetric polynomial in the roots of $f$, hence by the fundamental theorem on symmetric polynomials, $D(t)$ is a polynomial in the elementary symmetric functions, with coefficients in $\dZ$.  The elementary symmetric functions are the coefficients of $f$ viewed as a polynomial in $x$.  Thus $D(t) \in \dZ[t]$.)   Now consider $f'(t,x)=nx^{n-1}+\cdots \in \dZ[t,x]$, which is not monic for $n>1$.  Consider $n^{n-2}f'(t,x)=n^{n-1}x^{n-1}+\cdots=:{\underline f'}(t,nx)$.  The polynomial ${\underline f'}(t,y)\in \dZ[t,y]$ is monic in $y$.  For $\alpha$ algebraic over $\dQ(t)$ we have $f'(t,\alpha)=0 \Leftrightarrow {\underline f'}(t,n\alpha)=0 $.  Let $\alpha_1,...,\alpha_{n-1}$ be the roots of $f'(t,x)$.  Then $D({\underline f'}(t,x))=\prod_{i<j}(n\alpha_i-n\alpha_j)^2= n^{(n-1)(n-2)}\prod_{i<j}(\alpha_i-\alpha_j)^2=
 n^{(n-1)(n-2)}D(f'(t,x))$.

\vskip 0.5em

The main result is as follows.  \vskip 0.5em
 \begin{thm} \label{thm:inertia}
Assume the notations and assumptions above. \vskip 0.5em

1. Assume \vskip 0.5em
(i) The greatest common divisor of $D(f(t,x))$ and $(n-1)nD({\underline f'(t,x)})=:N \in \dZ$. \vskip 0.5em
(ii) There exists $c \in \dZ$ such that the only primes dividing both $D(f(c,x))$ and $(n-1)nD({\underline f'(c,x)})$ are prime divisors of $N$.\vskip 0.5em

Then there exists an arithmetic progression of elements $c \in \dZ$ such that in the splitting field of $f(c,x)$, all inertia groups in the corresponding Galois group at primes not dividing $N$ are of order $\leq 2$.  In particular, when $N=1$,  all inertia groups are of order $\leq 2$.
\vskip 0.5em
2. Assume $f(t,x)\in \dZ[t,x]$ and there exists an arithmetic progression of elements $c \in \dZ$ such that in the splitting field of $f(c,x)$, all inertia groups in the corresponding Galois group at primes not dividing $N$ are of order $\leq 2$ (i.e. the conclusion of part 1 above holds). Assume further:

(iii) $f(t,x)$ is irreducible in $\dQ[t,x]$ with Galois group $G $ over $\dQ(t)$.  \vskip 0.5em
 (iv) The splitting field of $f(t,x)$ over $\dQ(t)$ is regular over $\dQ$.\vskip 0.5em

Then there exist infinitely many $c \in \dZ$ such that $f(c,x)$ has Galois group $G$ over $\dQ$, the corresponding splitting fields are linearly disjoint over $\dQ$, and all inertia groups at primes not dividing $N$ are of order $\leq 2$.
\end{thm}\vskip 0.5em

 \begin{proof} 1. By Lemma \ref{thm:schinzel}, there exists an arithmetic progression $\cP=\{a+bt:t \in \dZ \}$ in $\dZ$ ($b\neq 0$) such that for every $c\in \cP$, the only primes dividing both $D(f(c,x))$ and $(n-1)nD({\underline f'(c,x)})$ are prime divisors of $N$.
 From the equation $D({\underline f'}(t,x))=n^{(n-1)(n-2)}D(f'(t,x))$ above, we have $D({\underline f'}(c,x))=n^{(n-1)(n-2)}D(f'(c,x))$, so {\it a fortiori} the only primes dividing both $D(f(c,x))$ and $(n-1)nD(f'(c,x))$ are prime divisors of $N$.
In particular the prime $2$ does not divide the discriminant of the corresponding splitting field $K$ and is therefore unramified in $K$.  We show next that all inertia groups at primes not dividing $N$ have \it exponent \rm at most $2$.  Fix $c$ and assume $p$ is a prime not dividing $N$ which ramifies in the splitting field $K$ of $f(c,x)$.  Then $p$ does not divide $(n-1)n$, for otherwise, $p$ would divide both $(n-1)n$ and  $D(f(c,x))$, hence also $N$, contrary to hypothesis.  In particular, every root of $f'(c,x)$ and of  $f''(c,x)$ is integral at $p$, as the leading coefficient of $f'(c,x)$ is $n$ and the leading coefficient of $f''(c,x)$ is $(n-1)n$.  Let $\fp$ be a prime of $K$ dividing $p$.
Let $\sigma$ be an element of the inertia group of $\fp$ of order greater than two.  Then for some root $\beta$ of $f(c,x)$, the orbit of $\beta$ under $\sigma$ has length at least three.    As $\sigma$ belongs to the inertia group of $\fp$, we have
$\beta\equiv \sigma(\beta)\equiv \sigma^2(\beta)$ mod $\fp.$  It follows that $f(c,x)$ factors mod $\fp$ as $(x-\beta)^3h(x)$ mod $\fp$.  This implies in particular that $D(f(c,x)$ and $D(f'(c,x)$ are divisible by $\fp$, and being both rational integers, they are divisible by $p$.  This contradicts the assumption that $p$ does not divide $N$.
  It follows that the exponent of every nontrivial inertia group at primes $p$ not dividing $N$ is $2$.
This, together with tame ramification of $K$ above, implies that all such nontrivial inertia groups are cyclic of order $2$.
 \vskip 0.5em
 2.  Let $\cP=\{a+bt:t \in \dZ \}$ be the given arithmetic progression.
 Consider the polynomial  $$g(t,x):=f(a+bt,x).$$
  Conditions (iii),(iv) hold for $g$ in place of $f$ as the substitution $t\mapsto a+bt$ defines an automorphism of  $\dQ(t)$.  By Hilbert's irreducibility theorem, there exist infinitely many specializations $t=c' \in \dZ$ for which $g(c',x)$ is irreducible with Galois group $G$ and the corresponding splitting fields are linearly disjoint over $\dQ$.  But $g(c',x)=f(a+bc',x)$, hence we have infinitely many integers $c=a+bc' \in \cP$ for which $f(c,x)$ is irreducible with Galois group $G$ and the corresponding splitting fields are linearly disjoint over $\dQ$, and in addition,
  for each such $c$, the corresponding Galois group has all inertia groups at primes not dividing $N$ of order at most two.
  \end{proof}

\begin{rem} \label{unramified} Given any realization $K/\dQ$ of $G$ with all inertia groups of order $2$ (as happens in Theorem \ref{thm:inertia} with $N=1$), there exists a quadratic extension $E/\dQ$ such that $G(KE/E)\cong G$ and $KE/E$ is unramified at all primes (finite and  infinite).\end{rem}
 \begin{proof}  Taking $E$ to be
  $\dQ(\sqrt{a})$ with $a$ negative, divisible by the product of the primes ramifying in $K$, and disjoint from $K$ over $\dQ$, Abhyankar's Lemma implies that $E$ swallows the ramification at all primes (including infinity) and $G(KE/E)\cong G$.
  \end{proof}
 \vskip 0.5em
\begin{rem} \label{exponent2_ab} Any group $G$ realizable over $\dQ$ with all nontrivial inertia groups of exponent $2$, 
is necessarily generated by elements of order $2$.
\end{rem}
\begin{proof} Suppose $K/\dQ$ is Galois with group $G$ and all  nontrivial inertia groups are of exponent $2$. The subgroup $H$ of $G$ generated by all these inertia groups has a fixed field  which is unramified at all finite primes,
 so $H=G$ by Minkowski's theorem.  
Furthermore, $H$ is clearly generated by elements of order $2$.
\end{proof}

\it Illustrative example. \rm The simplest nontrivial illustration of a group $G$ in Theorem \ref{thm:inertia} is the symmetric group $S_3$ of order $6$.  Let $f(t,x)=x^3+tx+t-1$.
 Specializing to $t=-1$ shows $f(-1,x)$ to be irreducible over $\dQ$, and factoring $f(-1,x)$ mod $5$ and mod $11$ yields Galois group $S_3$ for $f(-1,x)$ over $\dQ$, hence also for $f(t,x)$ over $\dQ(t)$.  Thus (i) holds. As for (ii), the discriminant $D(f(t,x))$ of $f(t,x)$ is $-4t^3-27(t-1)^2$.  If the splitting field of $f(t,x)$ over $\dQ(t)$ were not regular, then the Galois group over $\dC(t)$ would not be $S_3$, from which it would follow that the discriminant $-4t^3-27(t-1)^2$ is a square in $\dC(t)$.  However, the discriminant of $-4t^3-27(t-1)^2$ with respect to $t$ is $-629856\neq0$, so $-4t^3-27(t-1)^2$ factors over $\dC$ into the product of three distinct linear factors, hence is not a square in $\dC(t)$.  For (iii), we have $(3-1)3D({\underline f'(t,x)})=6\cdot(-12t)$, which is coprime to $-4t^3-27(t-1)^2$ in $\dZ[t]$.  Finally, taking $c=2$, we have $D(f(2,x))=59$, which is coprime to $6\cdot(-12\cdot 2)=-2^43^2$, verifying condition (iv).
 \vskip 0.5em

\section{$Examples$}\label{sec:examples} \rm
 In this section we give three examples of finite simple groups to which Theorem \ref{thm:inertia} is applied.
 \vskip 0.5em
 \begin{thm} \label{thm:A5exp2} The alternating group $A_5$ occurs infinitely
often as a Galois group over the rationals $\dQ$ with all nontrivial
inertia groups of order $2$.
\end{thm}

\begin{proof}

  Consider the generic quintic polynomial for $A_5$ \cite{HT}: $$f(u,v,x)=x^{5}+ux^{4}+(-6u-10)x^{3}+vx^{2}+(-u^{2}+12u+25-3v)x+9v-24+u^{3}+24u^{2}+27u$$  Let $f'(u,v,x)$ be the derivative of $f$ with respect to $x$.  Let ${\underline f'}(u,v,y)$ be defined by $5^3 f'={\underline f'}(u,v,5x)$.
Specialize $u$ to $0$ to obtain $$f(0,v,x)=x^{5}-10x^{3}+vx^{2}+(25-3v)x+9v-24$$
$$D(f)=-27v^{3}+6480v^{2}-28400v+24000$$
and $${\underline f'}(0,v,x)=x^{4}-150x^{2}+50vx+5^{5}-375v$$
$$D({\underline f'})=-2^{4}5^{6}(675v^{4}+396792v^{3}+1846800v^{2}-7920000v-20000000)$$
where $D(f)$ and $D({\underline f'})$ are the discriminants of $f(0,v,x)$ and ${\underline f'}(0,v,x)$ with respect to $x$.

 We now apply Theorem \ref{thm:inertia}, which implies that there are infinitely many specializations of $v\in\dZ$ which yield disjoint splitting fields of $f(0,v,x)$ with Galois group $A_5$ and such that all nontrivial inertia groups have exponent two, provided the following conditions hold:
\vskip 0.5em
(i) $f(0,v,x)$ is irreducible in $\dQ[v,x]$ with Galois group $A_5$ over $\dQ(v)$.  \vskip 0.5em
 (ii) The splitting field of $f(0,v,x)$ over $\dQ(v)$ is regular over $\dQ$.\vskip 0.5em
 (iii) $D(f(0,v,x))$ and $20D({\underline f'(0,v,x)})$ are coprime in
$\dZ[v]$. \vskip 0.5em
(iv) There exists $c \in \dZ$ such that $D(f(0,c,x))$ and $20D({\underline f'(0,c,x)})$ are coprime in
$\dZ$.\vskip 0.5em
We verify these conditions for $f(0,v,x)$:
\vskip 0.5em
(i) $f(0,v,x)$ is a specialization of a polynomial with Galois group $A_5$, so the Galois group over $\dQ(v)$ is a subgroup of $A_5$. For $v=0$ we obtain the polynomial $f(0,0,x)=x^{5}-10x^{3}+25x-24$ which has  Galois group $A_5$ over $\mathbb{Q}$. It follows that $f(0,v,x)$  has Galois group $A_5$ over $\dQ(v)$.  \vskip 0.5em
(ii) Let $K$ be the splitting field of $f(0,v,x)$ over $\dQ(v)$, and set $L=K\cap \widetilde{\dQ}(v)$. $L/\dQ(v)$ is Galois since $K$ and $\widetilde{\dQ}(v)$ are Galois over $\dQ(v)$. Moreover, $A_5=Gal(K/\dQ(v))$ is a simple group and thus we must have $L=\dQ(v)$ or $L=K$. It remains to show that $L\neq K$. If $L=K$ then $K\subseteq \widetilde{\dQ}(v)$ implying that the field discriminant of $K/\dQ(v)$ is a constant, which implies that the polynomial discriminant $D(f)$ of $f(0,v,x)$ is a constant times a square in $\dQ(v)$.  However $D(f)=-27v^{3}+6480v^{2}-28400v+24000$ is evidently not such.  It follows that $L=\dQ(v)$.
\vskip 0.5em
(iii) $D(f(0,v,x))$ and $20D({\underline f'(0,v,x)})$ are irreducible and thus coprime as polynomials. Moreover, the GCD of the coefficients of $D(f(0,v,x))$ is 1 and thus the polynomials are coprime in
$\dZ[v]$. \vskip 0.5em
(iv) For $c=-3$ we obtain $D(f(0,-3,x))=3^{2}\cdot8243^{2}$ and $20D({\underline f'(0,-3,x)})=-2^{4}\cdot 5^{6}\cdot 23^{2}\cdot 18379$.
\end{proof}
\vskip 0.5em
\vskip 0.5em

\begin{thm}
\label{exp2psl27}
$PSL_2(7)$ occurs infinitely
often as a Galois group over the rationals $\dQ$ with all nontrivial
inertia groups of order $2$. Moreover, these $PSL_2(7)$-extensions can be chosen to be tamely ramified (i.e., unramified at $2$) and contained in $\rr$.
\end{thm}
As in Remark \ref{unramified}, we obtain:
\begin{cor}
There exist infinitely many $PSL_2(7)$-Galois extensions $F|\qq$ such that, for some real quadratic number field $k$, the extension $Fk|k$ is an everywhere unramified $PSL_2(7)$-extension. \rm (Here the term ``everywhere unramified" includes the infinite primes of $k$).
\end{cor}

\begin{proof}[Proof of Theorem \ref{exp2psl27}:]
The following $4$-parameter family was given by Malle in \cite{Malle}:

$f:=x^7 - ((c - 2)a + 2b + c)x^6 + (-(b - 4)(c - 1)a^2 + ((c - 2)b^2
+(2c^2 - 5c + 4)b - 2c^2)a + b(2bc + 2c^2 + b^2 ))x^4
+((2c^2 - 1)(b - 4)a^2 + ((-2c^2 + c + 2)b^2 + (5c^2 + 2c - 4)b - 4c^2)a
-(c + 1)b^3 - c(2c + 3)b^2 + c^2b)x^3 + ((c^2 + 3c - 1)(4 - b) a^2
+((3c - 2)b^2 - 2(c^2 + 4c - 2)b + 4c^2 )a + b(b^2 + 3bc - c^2 ))cx^2
+(2abc - 8ac + ab - 4a - b^2 + 2bc)ac^2x - a^2(b - 4)c^3
+tx^2(x - c)(x^2 - bx + b)$.

It was shown there that $f$ has Galois group $PSL_2(7)$ over $\qq(a,b,c,t)$.

Note that as a family of regular Galois extensions over $\qq(t)$, the branch cycle structure is generically (i.e.\ for a dense open subset of all parameter choices for $a,b,c$) given by six involutions in $PSL_2(7)$.

After specialization $c\mapsto 1$, the discriminant of $f$ factors over $\zz$ as
$a^2\cdot g(t)^2$ for some irreducible polynomial $g\in \zz[a,b][t]$ of degree $5$ in $t$.

Similarly, the normalized discriminant $\Delta(\underline{f'}(t,x))$ factors as
$2^6\cdot 7^{20}\cdot h(t)$ for some irreducible polynomial $h\in \zz[a,b][t]$ of degree $9$.

This gives us many possible choices for specializations of $a$ and $b$ such that Theorem \ref{thm:inertia} can be applied. For simplicity, we restrict ourselves to one good choice of parameter values. So let $a:=1$ and $b:=5$ (and $c:=1$ as above). From now on, refer to this specialized polynomial by $f (=f(t,x))$. It is easy to verify that the polynomial $f$ still has Galois group $PSL_2(7)$ over $\qq(t)$.

To be explicit, we have
$$f(t,x)=x^7 - 10x^6 + 163x^4 - 333x^3 + 191x^2 - 12x - 1+ x^2(x-1)(x^2-5x+5)t. $$
Also, the discriminants $\Delta(f(t,x))$ and $7\cdot 6 \cdot \Delta(\underline{f'}(t,x))$ are coprime and remain so for at least one integer specialization $t\mapsto t_0$ (one may just choose $t_0=0$, which leaves the constant coefficients, factoring as $$23^2 \cdot 254106319^2$$ and
$$2^7\cdot 3 \cdot 7^{20}\cdot 1213\cdot 20789\cdot 208589 \cdot 592191293$$ respectively).

Furthermore, one verifies easily that the polynomial $f$ remains separable modulo $2$ for all integer specializations $t\mapsto t_0\in \zz$ (in fact, even irreducible). This means that the corresponding number field extensions of $\qq$ are unramified at $2$. Theorem \ref{thm:inertia} therefore yields an arithmetic progression of integers $t_0$ such that the splitting field of $f(t_0,x)$ over $\qq$ is (unramified at $2$ and) all inertia subgroups at ramified primes are of order $2$.
By Hilbert's irreducibility theorem, such an arithmetic progression even yields infinitely many linearly disjoint extensions, all with group $PSL_2(7)$ over $\qq$.

Moreover, for $t\mapsto t_0 \in \zz$ sufficiently small (i.e.\ negative of sufficiently large absolute value), the splitting field of $f(t_0,x)$ is contained in $\rr$. It is sufficient to verify this for one sufficiently small value and to note that the number of real roots of $f(t_0,x)$ can only change when one passes through a real branch point of the splitting field of $f$ over $\qq(t)$. But the specializations $t\mapsto t_0\in \zz$ from which we obtained our $PSL_2(7)$-extensions above form an arithmetic progression. Therefore, they include infinitely many $t_0$ which are sufficiently small to enforce the extension to be contained in $\rr$.
\end{proof}


 We treat a further non-abelian simple group and demonstrate how constant common divisors of $D(f(t,x))$ and $(n-1)nD({\underline f'(t,x)})$ in Theorem \ref{thm:inertia} may be dealt with in concrete examples.

We first give a new polynomial for the Galois group $PSL_3(3)$.
\begin{lem}
\label{lem:psl33_pol}
Let {\footnotesize $$f(t,x):=x(x^3 + 3x^2 - 27)(x^9 - 54x^6 - 81x^5 + 108x^4 + 837x^3 - 972x - 432)$$
$$-t(4x^{12} + 8x^{11} - 101x^{10} - 222x^9 + 428x^8 + 1970x^7 - 1020x^6 - 3240x^5
    - 1088x^4 - 672x^3 + 7776x + 5184).$$}
Then $f$ has regular Galois group $PSL_3(3)$ over $\qq(t)$.
\end{lem}
\begin{proof}
Set $G:={\rm{Gal}}(f|\qq(t))$.
From the factorization of $f(0,x)$ into irreducibles of degree $1$, $3$ and $9$, it follows that $|G|$ is divisible by $9$. This shows that $G$ is not a subgroup of $C_{13}\rtimes C_{12}$, which already leaves only three transitive groups of degree $13$, namely $PSL_3(3)$, $A_{13}$ and $S_{13}$. To show that $G$ is none of the latter two, the standard strategy (used for similar Galois group verifications e.g.\ in \cite{Malle} or \cite{Koe17}) is to use the fact that $PSL_3(3)$, unlike $S_{13}$ or $A_{13}$, has two non-conjugate subgroups of index $13$ (the stabilizers of a line and of a plane in $(\mathbb{F}_3)^3$).\\
Denote by $f_0(x)$ and $f_1(x)$ the coefficients of $f$ at the monomials $t^0$ and $t^1$ respectively.
Set \\
$g_0(Y):=27(Y^3 + 54Y^2 + 243Y - 4374)(Y^9 + 135Y^8 + 7776Y^7 + 267543Y^6 + 6613488Y^5 + 126128664Y^4 + 1734623424Y^3 + 14922863280Y^2 + 66119763456Y + 123974556480)$,\\
$g_1(Y):=(Y^4 + 54Y^3 + 972Y^2 + 18954Y + 104976)(Y^9 + 122Y^8 + 6264Y^7 + 205092Y^6 + 5385852Y^5 + 104188680Y^4 + 1218062772Y^3 + 8804914488Y^2 + 60533255664Y + 315446371488)$.

One verifies, e.g.\ with Magma, that the polynomial $f_0(x)g_1(Y)-f_1(x)g_0(Y)\in \qq[x,Y]$ is reducible, with factors of degree $4$ and $9$ in $x$. In other words, $f(t,x)=f_0(x)-tf_1(x)$ splits into factors of degree $4$ and $9$ over a root field of $g_0(Y)-tg_1(Y)$. Thus, there exists a degree-$13$ extension of $\qq(t)$ (necessarily contained in a splitting field of $f(t,x)$) over which $f$ is reducible, but does not have a root. This shows the existence of an index-$13$ subgroup in $G$, not fixing a point. Therefore, $G=PSL_3(3)$.\\
The regularity in the assertion follows immediately, since $PSL_3(3)$ is simple.
\end{proof}

The above polynomial $f(t,x)$ was obtained in the following way: In \cite[Theorem 7.2]{Koe17}, a multi-parameter polynomial $f(\alpha,\beta,t,x)$ of degree 13 in $x$ and with Galois group $PSL_3(3)$ was given. To obtain our polynomial, we specialized $\beta:=1$, $t:=-1$, and then made use of the remark after Theorem 7.2 in \cite{Koe17} asserting that the branch cycle structure of $f(\alpha,\beta,t,x)$ with respect to $\alpha$ consists of six involutions in $PSL_3(3)$, all of cycle type $(2^4.1^5)$. In particular, it then follows from the Riemann-Hurwitz genus formula that a root field of $f(\alpha,1,-1,x)$ over $\qq(\alpha)$ is of genus zero, and is in fact a rational function field (the last implication follows since the degree of the extension is odd). A standard Riemann-Roch space computation (using Magma) then yielded a parameter $\xi$ for this genus zero function field, whose minimal polynomial (after renaming $\alpha$ to $t$ and performing some linear transformations) is exactly the polynomial $f(t,x)$ above.

\begin{thm}
\label{thm:psl33}
$PSL_3(3)$ occurs infinitely
often as a Galois group over the rationals $\dQ$ with all nontrivial
inertia groups of order $2$.
\end{thm}
\begin{proof}
Consider the polynomial $f(t,x)$ from Lemma \ref{lem:psl33_pol}.
One computes that the discriminant of $f(t,x)$ equals
$D(f(t,x))=2^{18}\cdot 3^{12}\cdot 12491^6\cdot (36t^2 - 40t - 27)^4(31171328t^4 - 8088768t^3 - 279653877t^2 + 125341344t + 48892572)^4$, the common prime divisors with $(n-1)nD({\underline f'(t,x)})$ are only the constants $2$, $3$ and $12491$, and this remains true upon e.g.\ evaluation $t\mapsto 0$. It then follows from Theorem \ref{thm:inertia} that there exists some arithmetic progression $a+R\zz$, with $R\in \nn$ coprime to $2,3$ and $12491$, such that in a splitting field $f(a+Rz,x)$, all inertia groups at primes $p\notin\{2,3,12491\}$ have order $\le 2$ (for all $z\in \zz$).

On the other hand, there are many specialization values $t\mapsto b\in \zz$ such that a splitting field of $f(b,x)$ has discriminant coprime to all of $2,3$ and 12491. E.g., $b=1$ yields field discriminant $(23\cdot 31\cdot 109\cdot 23843)^4$ for a root field. Krasner's lemma then implies that for all $t_0\in \zz$ sufficiently close $p$-adically to $1$ for all $p\in \{2,3,12491\}$ (i.e.\ $t_0$ in some arithmetic progression $1+M\zz$, with $M$ only divisible by the three primes $2,3$ and $12491$), the splitting field of $f(t_0,x)$ will be unramified at these three primes as well.

The Chinese remainder theorem then yields an arithmetic progression of integer values $t_0$ such that in a splitting field of $f(t_0,x)$, all inertia groups have order $\le 2$ (and in fact always order $1$ for the three primes $2,3,12491$). Theorem \ref{thm:inertia} now gives the desired result.
\end{proof}				
\section{Optimally intersective  realizations}\label{sec:intersective} \rm
\def\cprime{$'$}
In this section we show how Theorem \ref{thm:inertia} yields the existence of infinitely many optimally intersective realizations of $A_5$ and $PSL_3(3)$ over $\dQ$.  This is already known for $PSL_2(7)$ \cite{jk}.
\vskip 0.5em
A monic polynomial in one variable with rational integer coefficients is called {\it intersective} if it has a root modulo $m$ for all positive integers $m$, and {\it nontrivially intersective} if in addition it has no rational root.  In this paper ``intersective"   means ``nontrivially intersective".  Let $G$ be a finite noncyclic group and let $r(G)$ be the smallest number of irreducible factors of an intersective
 polynomial with Galois group $G$ over $\dQ$.
There is a group-theoretically defined lower bound for $r(G)$, given by the
 smallest number $s(G)$ of proper subgroups of $G$ having the property that the union of the
 conjugates of those subgroups is $G$ and their intersection is trivial.  This follows from
 \vskip 0.5em
 \begin{prop} \label{thm:char} {\em (\cite [Prop. 2.1]{so})}
Let $K/\dQ$ be a finite
Galois extension with noncyclic Galois group $G$.  The following are
equivalent:

(1) \ \ $K$ is the splitting field of a product $f=g_1
\cdots g_m$ of $m$ irreducible polynomials of degree greater than
$1$ in $\dQ[x]$ and $f$ has a root in $\dQ_p$ for all (finite) primes $p$.

(2) \ \  $G$ is the union of the conjugates of $m$ proper
subgroups $H_1,...,H_m$, the intersection of all these conjugates
is trivial, and for all (finite) primes $\fp$ of $K$, the decomposition
group $G(\fp)$ is contained in a conjugate of some $H_i$.
\end{prop}

\vskip .5em
We call an intersective polynomial $f(x)$ with Galois group $G$ over $\dQ$ {\it optimally intersective for}  $G$ if the above lower bound $s(G)$ for $r(G)$ is attained, i.e. $f(x)$ is the product of $s(G)$ irreducible factors.  Accordingly, a Galois extension $K/\dQ$ with Galois group $G$ is called an {\it optimally intersective realization of} $G$ if it is the splitting field of a polynomial which is  optimally intersective for  $G$.  Among the (noncyclic) alternating groups $A_n$, $s(A_n)=2$ if and only if $n \in \{4,5,6,7,8\}$ see \cite{Bubboloni}.  Single examples of optimally intersective polynomials for $A_n$ for these values of $n$ have been found \cite{RS}.
It is also  known that there exist infinitely many disjoint optimally intersective realizations of $A_4$ \cite{spearman}.  This also follows from \cite{so} as a special case ($A_4$ being solvable), but the proof in \cite{spearman} is more direct and explicit, using specializations of a parametric polynomial for $A_4$.  We now use Theorem \ref{thm:inertia} to prove that the same result holds for $A_5$.

\vskip .5em
\begin{thm} \label{thm:intersectiveA5} There exist infinitely many disjoint optimally intersective realizations of $A_5$.
\end{thm}
The proof uses the following
\begin{lem} \label{inertiaC3} To prove Theorem \ref{thm:intersectiveA5} it suffices to prove that there exist infinitely many disjoint Galois realizations of $A_5$ over $\dQ$ such that at no prime does $S_3$ appear as a decomposition group. \end{lem}

\begin{proof}
 For simple groups such as $A_5$, the trivial intersection condition on the conjugates of the subgroups $H_1,...,H_m$ in part (2) above, is superfluous, as it is satisfied automatically.  Thus without loss of generality the set $H_1,...,H_m$ may be assumed to consist of maximal subgroups of  $G$ if $G$ is simple.  Now $A_5$ is the union of the conjugates of the dihedral group $D_5$ and the alternating group $A_4$.  We need to show that if $S_3$ appears  as a decomposition group (in a given $A_5$ Galois extension of $\dQ$) at no prime, then every decomposition group is contained in a conjugate of one of the subgroups $D_5$ or $A_4$.  We first observe that every decomposition group is necessarily solvable (every finite Galois extension of a local field is solvable), hence every decomposition group in (the nonsolvable group) $A_5$ is proper.  Let $H$ be a proper subgroup of $A_5$.  We may assume $H$ noncyclic, as every element of $A_5$ is contained in a conjugate of one of the subgroups $D_5$ or $A_4$. If $H$ has a fixed point, then it is contained in a conjugate of $A_4$, so assume $H$ has no fixed point.  If $H$ has order divisible by $5$, then it is $D_5$. The only remaining possibility is $S_3$, generated by the $3$-cycle $(123)$ and the product of two transpositions $(12)(45)$, up to conjugacy.  \end{proof}
\vskip 1em
\it Proof of Theorem \ref{thm:intersectiveA5}. \rm
We now apply Lemma \ref{inertiaC3} and claim that there are infinitely many specializations of $v\in\dZ$ which yield disjoint splitting fields of $f(0,v,x)$ with Galois group $A_5$ and such that at no prime does $S_3$ appear as a decomposition group.  Suppose $S_3$ appears as a decomposition group at some prime $p$ for some specialization of $v$.  Then $p$ ramifies and the inertia group is necessarily the unique normal subgroup $C_3$.  We now apply Theorem \ref{thm:A5exp2}, which implies that there are infinitely many specializations of $v\in\dZ$ which yield disjoint splitting fields of $f(0,v,x)$ with Galois group $A_5$ and such that all nontrivial inertia groups have exponent two. \qed

\begin{thm} \label{thm:intersectiveA5} There exist infinitely many disjoint optimally intersective realizations of $PSL_3(3)$.
\end{thm}
\begin{proof}
Observe that the only fixed point free elements in $PSL_3(3)$, in its natural degree-$13$ action, are the elements of order $13$. This implies that $PSL_3(3)$ is the union of conjugates of subgroups $H_1$ and $H_2$, where $H_1$ is a point stabilizer in the natural action, and $H_2$ is a $13$-Sylow subgroup.
In analogy to the previous section, it suffices to provide infinitely many disjoint $PSL_3(3)$-realizations over $\qq$ such that all decomposition subgroups are contained in a conjugate of $H_1$ or $H_2$. In particular, it suffices to provide these realizations such that all {\it non-cyclic} decomposition subgroups are contained in a conjugate of $H_1$.\\
Lemma \ref{lem:psl33_pol} provides such realizations. Namely, all decomposition groups at ramified primes in the realizations given there are contained in the normalizer of a subgroup of order $2$ (and of cycle type $(2^4.1^5)$). Since this normalizer in $PSL_3(3)$ fixes a point, the assertion follows.
\end{proof}

\begin{rem}
    The first author has independently proved by other methods the existence of an optimally intersective realization for infinitely many $2$-coverable nonsolvable linear groups, and  also the existence of  infinitely many optimally intersective realizations of the groups $PSL_3(2)\cong PSL_2(7)$, $AGL_3(2)$, and $M_{11}$ using parametric polynomials.  See \cite{jk}.
\end{rem}

\bibliographystyle{plain}

\begin{thebibliography}{10}

\bibitem{Bubboloni}
D. Bubboloni,
\newblock {\em Coverings of the  symmetric and alternating groups},
\newblock {\rm preprint, Quaderno del Dipartimento di Matematica U. Dini, Firenze 7 (1998).}
\newblock Available at {https://arxiv.org/pdf/1009.3866.pdf}



\bibitem{HT}
K. Hashimoto, H. Tsunogai,
\newblock {\em Generic polynomials over $\dQ$ with two parameters for the transitive groups of degree five},
\newblock {\rm Proc. Japan Acad. {\bf 79}, \rm Ser. A (2003), 142-145 }

\bibitem{JLY}
C. Jensen, A. Ledet, N. Yui,
\newblock {\em Generic polynomials. Constructive aspects of the inverse Galois problem},
\newblock {\rm Mathematical Sciences Research Institute Publications {\bf 45}, Cambridge, 2002.}


 \bibitem{ko}
T. Kondo,
\newblock  Algebraic number fields with the discriminant equal to that of a quadratic number
field,
\newblock {\em J. Math. Soc. Japan} 47: 31-36, 1995.



\bibitem{jk}
J. K$\ddot{\rm o}$nig,
\newblock {\em On intersective polynomials with nonsolvable Galois group}
\newblock (Preprint) https://arxiv.org/pdf/1605.07802v3.pdf

\bibitem{Koe17}
J.\ K\"onig,
\newblock Computation of Hurwitz spaces and new explicit polynomials for almost simple Galois groups,
\newblock {\em Math. Comp. 86: 1473-1498, 2017.}

\bibitem{Malle}
G.\ Malle,
\newblock Multi-parameter polynomials with given Galois group,
\newblock {\em J.\ Symb.\ Comput.\ } 21: 1--15, 2000.

\bibitem{spearman1}
M. J. Lavallee, B. K. Spearman and Q. Yang,
\newblock {\em Math. J. Okayama Univ.} {\bf 56} (2014), 27–33
\newblock Intersective polynomials with Galois Group $D_5$

\bibitem{spearman}
P.D. Lee, B.K. Spearman, Q. Yang,
\newblock {\em A Parametric Family of Intersective Polynomials with Galois Group $A_4$},
\newblock {\rm Comm. Alg. {\bf 43} (5) (2015), 1784-1790}

\bibitem{RS}
D. Rabayev, J. Sonn,
\newblock {\em On Galois realizations of the 2-coverable symmetric and alternating groups},
\newblock {\rm Comm. Alg. {\bf 42}, no. 1 (2014), 253-258.}

\bibitem{so}
J. Sonn,
\newblock {\em Polynomials with roots in $\dQ_p$ for all $p$},
\newblock {\rm Proc. Amer. Math Soc. {\bf 136}, no. 6 (2008) 1955--1960}

\bibitem{sch}
A. Schinzel,
\newblock A property of polynomials with an application to Siegel's lemma,
\newblock {\em Monatsh. Math. } 137: 239-251, 2002



\bibitem{VdW}
B.L. Van der Waerden,
\newblock Die Zerlegungs-und Tr\"{a}gheitsgruppe
als Permutationsgruppen,
\newblock {\em Math. Ann.} 111, 731-733, 1935.


\end{thebibliography}

\end{document}